\documentclass[12pt]{amsart}

\begin{document}
\numberwithin{equation}{section}

\def\1#1{\overline{#1}}
\def\2#1{\widetilde{#1}}
\def\3#1{\widehat{#1}}
\def\4#1{\mathbb{#1}}
\def\5#1{\frak{#1}}
\def\6#1{{\mathcal{#1}}}
\def\7#1{{\bf{#1}}}

\def\C{{\4C}}
\def\R{{\4R}}
\def\N{{\4N}}
\def\Z{{\4Z}}
\def\P{{\4P}}
\def\Q{{\4Q}}

\title[Multi-resonant biholomorphisms]{Dynamics of multi-resonant biholomorphisms}
\author[F. Bracci]{Filippo Bracci*}
\address{F. Bracci: Dipartimento Di Matematica, Universit\`{a} Di Roma \textquotedblleft Tor
Vergata\textquotedblright, Via Della Ricerca Scientifica 1,
00133, Roma, Italy. } \email{fbracci@mat.uniroma2.it}
\author[J. Raissy]{Jasmin Raissy**}
\address{J. Raissy: Dipartimento Di Matematica e Applicazioni, Universit\`{a} degli Studi di Milano Bi\-coc\-ca, Via Roberto Cozzi 53,
20125, Milano, Italy. } \email{jasmin.raissy@unimib.it}
\author[D. Zaitsev]{Dmitri Zaitsev***}
\address{D. Zaitsev: School of Mathematics, Trinity College Dublin, Dublin 2, Ireland}
\email{zaitsev@maths.tcd.ie}
%\subjclass{}
\thanks{$^{*}$Partially supported by the ERC grant ``HEVO - Holomorphic Evolution Equations'' n. 277691.}
\thanks{$^{**}$Partially supported by FSE, Regione Lombardia, and by the PRIN2009 grant ``Critical Point Theory and Perturbative Methods for Nonlinear Differential Equations''.}
\thanks{$^{***}$Partially supported by the Science Foundation Ireland grant 06/RFP/MAT018.}

%\tableofcontents

\def\Label#1{\label{#1}{\bf (#1)}~}
%\def\Label#1{\label{#1}}

% Standard sets

\def\cn{{\C^n}}
\def\cnn{{\C^{n'}}}
\def\ocn{\2{\C^n}}
\def\ocnn{\2{\C^{n'}}}

% Abbreviations

\let\no=\noindent
\let\bi=\bigskip
\let\me=\medskip
\let\sm=\smallskip
\let\ce=\centerline
\let\ri=\rightline
\let\te=\textstyle
\let\gd=\goodbreak
\let\io=\infty
\def\qqquad{\quad\qquad}

\def\dist{{\rm dist}}
\def\const{{\rm const}}
\def\rk{{\rm rank\,}}
\def\id{{\sf id}}
\def\aut{{\sf aut}}
\def\Aut{{\sf Aut}}
\def\CR{{\rm CR}}
\def\GL{{\sf GL}}
\def\Re{{\sf Re}\,}
\def\Im{{\sf Im}\,}
\def\span{\text{\rm span}}
\def\res{{\rm Res}\,}

\def\codim{{\rm codim}}
\def\crd{\dim_{{\rm CR}}}
\def\crc{{\rm codim_{CR}}}

\def\phi{\varphi}
\def\eps{\epsilon}
\def\d{\partial}
\def\a{\alpha}
\def\b{\beta}
\def\g{\gamma}
\def\G{\Gamma}
\def\D{\Delta}
\def\Om{\Omega}
\def\k{\kappa}
\def\l{\lambda}
\def\L{\Lambda}
\def\z{{\bar z}}
\def\w{{\bar w}}
\def\t{\tau}
\def\th{\theta}
\def\ta{\tilde{\alpha}}
%************************** SIDE-REMARKS *****************************
\def\sideremark#1{\ifvmode\leavevmode\fi\vadjust{%            The remark
\vbox to0pt{\hbox to 0pt{\hskip\hsize\hskip1em%               will appear only
\vbox{\hsize1.5cm\tiny\raggedright\pretolerance10000%          on the side
\noindent #1\hfill}\hss}\vbox to8pt{\vfil}\vss}}}%           in 3cm
%                                                            wide box

\def\Dif{{\sf Diff}(\C^n;0)}
\def\Diff{{\sf Diff}}

\emergencystretch15pt \frenchspacing

\newtheorem{theorem}{Theorem}[section]
\newtheorem{lemma}[theorem]{Lemma}
\newtheorem{proposition}[theorem]{Proposition}
\newtheorem{corollary}[theorem]{Corollary}

\theoremstyle{definition}
\newtheorem{definition}[theorem]{Definition}
\newtheorem{example}[theorem]{Example}

\theoremstyle{remark}
\newtheorem{remark}[theorem]{Remark}
\numberwithin{equation}{section}

\def\bl{\begin{lemma}}
\def\el{\end{lemma}}
\def\br{\begin{remark}}
\def\er{\end{remark}}
\def\bp{\begin{Pro}}
\def\ep{\end{Pro}}
\def\bt{\begin{Thm}}
\def\et{\end{Thm}}
\def\bc{\begin{Cor}}
\def\ec{\end{Cor}}
\def\bd{\begin{Def}}
\def\ed{\end{Def}}
\def\be{\begin{Exa}}
\def\ee{\end{Exa}}
\def\bpf{\begin{proof}}
\def\epf{\end{proof}}
\def\ben{\begin{enumerate}}
\def\een{\end{enumerate}}
\def\beq{\begin{equation}}
\def\eeq{\end{equation}}

\begin{abstract}
The goal of this paper is to study the dynamics of holomorphic
diffeomorphisms in $\C^n$ such that the resonances among the first $1\le r\le n$ eigenvalues of the
differential are generated over $\N$ by a finite number of $\Q$-linearly independent multi-indices (and more resonances are
allowed for other eigenvalues). We give sharp conditions
for the existence of basins of attraction where a Fatou coordinate can be defined.
Furthermore, we obtain a generalization of the Leau-Fatou flower theorem, providing a complete description of the dynamics in a full neighborhood of the origin for $1$-resonant parabolically attracting holomorphic germs in Poincar\'e-Dulac normal form.
\end{abstract}

\maketitle

\section{Introduction}

The goal of this paper is to put in a single frame both results by the first and the third author \cite{BZ} and by Hakim \cite{Ha, Ha1} on the existence of basins of attraction for germs of biholomorphisms of $\C^n$.
In \cite{BZ} basins of attraction were constructed, modeled on the Leau-Fatou flower theorem, for germs with {\em one-dimensional sets of resonances}; on the other hand, in \cite{Ha, Ha1} the basins were constructed for germs tangent to the identity, in which case the set of resonances is the set of all multi-indices.
In the present paper we deal with the case when the eigenvalues of the linear part have resonances generated by a finite number of multi-indices.
Whereas considerable research has been done for {\em one-dimensional sets of resonances}, much less is known when resonances have {\em several generators}, which is the subject of our study.

Let $F$ be a germ of biholomorphism of $\C^n$,
fixing the origin $0$, and whose differential $dF_0$ is diagonalizable with eigenvalues $\{\l_1,\ldots, \l_n\}$,
 %It is well-known that presence of resonances among the eigenvalues $\{\l_1,\ldots, \l_n\}$ of $dF_0$ prevents linearization of the germ, and hence makes the understanding of the dynamics of $F$ near $0$ very complicated (see \cite[Ch. IV]{Ar}, \cite{Ab2}).
%One of the main conclusions of \cite{BZ} establishes the existence of basins of attraction for germs with one-dimensional family of resonances, whereas
%When resonances exist, the dynamics can be rather different from what one would expect looking at the linear part of the germ. For instance, as shown in the recent paper \cite{BZ}, even when all eigenvalues have modulus one, but are not roots of unity, a parabolic-like dynamics might appear, and parabolic basins of attraction might show up.
%
%The case when the  resonances are generated by only one multi-index (the so-called $1$-resonance case) has been studied in \cite{BZ}, where it is proved that under a non-degeneracy condition and a parabolic-type attracting condition there exist basins of attraction modeled on a one-dimensional Leau-Fatou flower theorem.
and denote by $\Dif$ the space of germs of biholomorphisms of $\C^n$ fixing $0$. Given $F$, we shall say that it is {\sl $m$-resonant with respect to the first $r$ eigenvalues $\{\l_1,\ldots, \l_r\}$} if there exist $m$ linearly independent multi-indices $P^1,\ldots, P^m\in \N^r\times\{0\}^{n-r}$, such that all resonances with respect to the first $r$ eigenvalues, {\sl i.e.}, the relations $\l_s=\prod_{j=1}^n \l_j^{\beta_j}$ for $1\leq s\leq r$, are precisely of the form $(\beta_1,\ldots, \beta_n)=\sum_{t=1}^m k_t P^t+e_s$ with $k_t\in \N$, $e_s=(0,\ldots, 0,1,0\ldots, 0)$ with $1$ in the $s$-th coordinate (see \cite{Ra} for a detailed study of the general structure of resonances). Here and in the following we shall adopt the
notation $\N=\{0,1,\ldots\}$.

The classical Poincar\'e-Dulac theory (see \cite[Chapter IV]{Ar}, \cite{Ab2})  implies that $F$ is formally conjugated to a map $G=(G_1,\ldots, G_n)$ with
\[
G_j(z)=\l_j z_j + \sum_{| K {\bf P}|\ge 1 \atop K\in\N^m} a_{K,j} z^{K {\bf P}} z_j, \quad j=1,\ldots, r,
\]
where $K{\bf P}:=\sum_{t=1}^m k_t P^t$. Let $k_0$ be the {\sl weighted order} of $F$, {\sl i.e.,} the minimal $|K|$ such that $a_{K,j}\neq 0$ for some $1\leq j\leq r$. Such a number is a holomorphic invariant of $F$.

In the study of the dynamics of $F$, an important r\^ole is played by the map $\pi: \C^n \to \C^m$ given, in multi-indices notation, by $z\mapsto (z^{P^1},\ldots, z^{P^m})$. In fact, to $G$ as above there is a canonically associated $\Phi$ formal biholomorphism of $(\C^m, 0)$  such that $\pi \circ G=\Phi \circ \pi$. If $k_0<+\infty$, then $\Phi$ is of the form
\[
\Phi(u) := u + H_{k_0+1}(u) + O(\|u\|^{k_0+2}),
\]
where $H_{k_0+1}$ is a non-zero homogeneous polynomial of degree $k_0+1$.
The map $\Phi$ depends on the formal normal form $G$, but it is an invariant of $F$ up to conjugation. Its truncation
$$
f(u):=u+H_{k_0+1}(u)
$$
is called a {\sl parabolic shadow} of $F$, because, since $F$ is  holomorphically conjugated to $G$ up to any fixed order,  the foliation $\{z\in \C^n: \pi(z)=\hbox{const}\}$ is  invariant under $G$ up to the same fixed order, and the action---of parabolic type---induced by $G$ near $0$ on the leaf space of such a foliation is given by $f$ up to order $k_0+2$.

One can then suspect  that, if $f$ has a basin of attraction and some other ``attracting'' conditions controlling the dynamics of $F$ on the fibers of $\pi\colon\C^n\to \C^m$ are satisfied, then $F$ must have a basin of attraction with boundary at $0$. This is exactly the case, and it is our main result. Let $v\in \C^m\setminus\{0\}$. We say that a $F$, $m$-resonant with respect to the first $r$ eigenvalues, is {\sl $(f,v)$-attracting-non-degenerate} if $v$ is a non-degenerate characteristic direction for $f$ in the sense of Hakim \cite{Ha}, {\sl i.e.,} $H_{k_0+1}(v)=c v$ for some $c\ne0$, such that all directors of $v$ have  strictly positive real part (see Section \ref{preli} for precise definitions). Although $f$ is not uniquely determined by $F$, we will show that if $F$ is $(f,v)$-attracting-non-degenerate with respect to some parabolic shadow $f$ and some vector $v$, then for any other parabolic shadow $\2f$ of $F$ there exists a vector $\2v$ such that $F$ is $(\2f,\2v)$-attracting-non-degenerate. By Hakim's theory \cite{Ha}, \cite{Ha1} (see also \cite{Ari}), if $F$ is $(f,v)$-attracting-non-degenerate then the map $f$ admits a basin of attraction with $0$ on the boundary, and with all orbits tending to $0$ tangent to the direction $[v]$.

If $F$ is $(f,v)$-attracting-non-degenerate, we say that $F$ is {\sl $(f,v)$-parabolically attracting with respect to $\{\l_1,\ldots, \l_r\}$} if $H_{k_0+1}(v)= -(1/k_0) v$ (which can be obtained by scaling $v$) and
\[
\Re\Bigg(\sum_{|K|=k_0\atop K\in\N^m} {\frac{a_{K, j}}{\l_j}} v^{K}\Bigg) <0 \quad j=1,\ldots, r.
\]
Such a condition is invariant in the sense that if $F$ is $(f,v)$-parabolically attracting with respect to some parabolic shadow $f$ and some vector $v$, then for any other parabolic shadow $\2f$ of $F$ there exists a vector $\2v$ such that $F$ is $(\2f,\2v)$-parabolically attracting. We simply say that $F$ is {\sl parabolically attracting} if it is $(f,v)$-parabolically attracting with respect to some parabolic shadow $f$ and some vector $v$. Our main theorem is the following:

\begin{theorem}\label{mainintro}
Let $F\in\Dif$ be $m$-resonant with respect to the  eigenvalues
$\{\l_1,\ldots\l_r\}$ and of weighted order $k_0$. Assume that
$|\l_j|=1$ for $j=1, \dots, r$ and $|\l_j|<1$ for
$j=r+1,\ldots, n$. If $F$ is parabolically attracting, then
there exist (at least) $k_0$ disjoint basins of attraction
having $0$ at the boundary.

Moreover, for each basin of attraction $B$ there exists a holomorphic map $\psi: B \to \C$ such that for all $z\in B$
\[
\psi\circ F(z)=\psi(z)+1.
\]
\end{theorem}

Such a theorem on one side generalizes the corresponding result in \cite{BZ} for $1$-resonant germs with respect to the first $r$ eigenvalues, and on the other side shows the existence of a {\sl Fatou coordinate} for the germ on each basin of attraction. In \cite{BZ}  it is  shown that the non-degeneracy and parabolically attracting hypotheses are sharp for the property of having a basin of attraction.

Next we consider the case when $F$ is $m$-resonant with respect to \emph{all} the eigenvalues $\{\l_1,\ldots, \l_n\}$ and $|\l_j|=1$ for all $j=1,\ldots, n$. We show that, in this case, if $F$ is attracting-non-degenerate or parabolically attracting then so is $F^{-1}$. This allows us to show the existence of repelling basins for $F$, giving a generalization of the Leau-Fatou flower theorem in such a case. In fact, under these hypotheses, and in the $1$-resonant case, we prove that a germ $F$ which is holomorphically conjugated to a Poincar\'e-Dulac normal form admits a full punctured neighborhood of $0$ made of attracting and repelling basins for $F$ plus invariant hypersurfaces where the map is linearizable (see Theorem \ref{L-Fatou} for the precise statement).

\me

The plan of the paper is the following. In Section \ref{preli} we briefly recall the known results for the dynamics of maps tangent to the identity that we shall use to prove our main result. In Section \ref{3} we define $m$-resonant germs with respect to the first $r$ eigenvalues and study their basic properties. The proof of Theorem \ref{mainintro} is in Section \ref{4} (see Theorem \ref{mainthm} and Proposition \ref{fatou_m_res}). Finally, in Section \ref{final} we discuss our generalization of the classical Leau-Fatou flower theorem and we give an example showing that a $m$-resonant germ, $m\geq 2$, might be attracting-non-degenerate with respect to two different directions and  parabolically attracting with respect to one but not to the other.

\medskip

We thank the referee for many comments and remarks which improved the original manuscript.

\section{Preliminaries on germs tangent to the identity in
$\C^m$}\label{preli}

Let
\begin{equation}\label{h}
h(u) := u + H_{k_0+1}(u) + O(\|u\|^{k_0+2}),
\end{equation}
be a germ at $0$ of a holomorphic diffeomorphism of $\C^m$,
$m\ge1$, tangent to the identity, where  $H_{k_0+1}$ is the
first non-zero term in the homogeneous expansion of $h$, and
$k_0\ge 1$. We call the number $k_0+1\ge 2$ the {\it order} of
$h$.

If $m=1$ we have $H_{k_0+1}(u)=A u^{k_0+1}$ and the {\sl
attracting directions} $\{v_1,\ldots, v_k\}$ for $h$ are defined as the $k_0$-th roots of $-\frac{|A|}{A}$. These are precisely
the directions $v$ such that the term $Av^{k+1}$ is in the
direction opposite to $v$. An {\sl attracting petal} $P$ for
$h$ is a simply-connected domain such that $0\in
\partial P$, $h(P)\subseteq P$ and $\lim_{\ell\to\infty}h^{\circ
\ell}(z)=0$ for all $z\in P$, where $h^{\circ \ell}$ denotes the
$\ell$-th iterate of $h$. The attracting directions for $h^{-1}$
are called {\sl repelling directions} for $h$ and the
attracting petals for $h^{-1}$ are {\sl repelling petals} for
$h$.

We state here  the Leau-Fatou flower theorem (see, {\sl e.g.},
\cite{Ab2}, \cite{Br}). We write $a\sim b$ whenever there exist
constants $0<c<C$ such that $ca\le b\le Ca$.

\begin{theorem}[Leau-Fatou]\label{LF}
Let $h(u)$ be as in \eqref{h} with $m=1$. Then for each attracting
direction $v$ of $h$ there exists an attracting petal $P$ for
$h$ (said {\sl centered at $v$}) such that for each $z\in P$
the following hold:
\begin{enumerate}
  \item $h^{\circ \ell}(z)\ne0$ for all $\ell$ and $\lim_{\ell\to\infty}\frac{h^{\circ \ell}(z)}{|h^{\circ
  \ell}(z)|}=v$,
  \item $|h^{\circ \ell}(z)|^{k_0}\sim \frac{1}{\ell}$.
\end{enumerate}
Moreover,  the set given by the union of all $k_0$ attracting petals and $k_0$
repelling petals for $h$ forms a punctured open neighborhood of
$0$.
\end{theorem}

By the property (1), attracting (resp. repelling) petals
centered at different attracting (resp. repelling)
directions must be disjoint.

\sm
For $m>1$ the situation is more complicated and a complete
description of the dynamics in a full neighborhood of the
origin is still unknown (with the exception of the substantial class of bidimensional examples studied in \cite{AT2}). In this paper we shall use Hakim's results, that we are going to recall here.

\begin{definition}\label{chardir}
Let $h\in\Diff(\C^m, 0)$ be of the form \eqref{h}.  A {\it
characteristic direction} for $h$ is a non-zero vector
$v\in\C^m\setminus\{O\}$ such that $H_{k_0+1}(v)=\lambda v$ for
some $\lambda\in\C$. If $H_{k_0+1}(v)=0$, $v$ is a {\it
degenerate} characteristic direction; otherwise, (that is, if
$\lambda\ne 0$) $v$ is {\it non-degenerate}.
\end{definition}

\sm There is an equivalent definition of characteristic
directions. The $m$-tuple of $(k_0+1)$-homogeneous
polynomials~$H_{k_0+1}$ induces a rational self-map
of~$\C\P^{m-1}$, denoted by~$\widetilde H_{k_0+1}$. Then, under the
canonical projection $\C^m\setminus\{0\}\to\C\P^{m-1}$,
non-degenerate characteristic directions correspond exactly to
fixed points of~$\widetilde H_{k_0+1}$, and degenerate characteristic
directions correspond  to indeterminacy points of~$ \widetilde H_{k_0+1}$.
Generically, there is only a finite number of characteristic
directions, and using Bezout's theorem it is easy to prove
(see, {\sl e.g.}, \cite[Lemma 2.1]{AT}) that this number,
counted with multiplicity, is given
by~$((k_0+1)^m-1)/k_0$.

\begin{definition}\label{directors}
Let $h\in\Diff(\C^m, 0)$ be of the form \eqref{h}. Given a
non-degenerate characteristic direction $[v]\in\C\P^{m-1}$ for
$h$, the eigenvalues $\alpha_1,\ldots,\alpha_{m-1}\in\C$ of the
linear operator
$$
A(v):=\frac{1}{k_0}(d(\widetilde H_{k_0+1})_{[v]}-\id) \colon T_{[v]}\C\P^{m-1}\to T_{[v]}\C\P^{m-1}
$$
are the called the {\it directors} of~$[v]$. If all the directors
of $[v]$ have strictly positive real parts, we call $[v]$ a
{\sl fully attractive} non-degenerate characteristic direction
of $h$.
\end{definition}

\br\label{normal-repres}
Let $[v]$ be a  non-degenerate characteristic direction of $h$. Let $H_{k_0+1}(v)=\tau v$. Then,
replacing $v$ by  $(-k_0\tau)^{-1/k_0}v$, we can assume
\begin{equation}\label{1overk}
H_{k_{0}+1}(v)=-\frac1{k_{0}} v.
\end{equation}\er

\begin{definition}
A representative $v$ of $[v]$ such that \eqref{1overk} is
satisfied is called {\sl normalized}.
\end{definition}

Note that  a normalized representative is uniquely determined
up  to multiplication by $k_{0}$-th roots of unity.

\begin{remark}\label{1-nota-fully-attractive}
For $m=1$ each germ $h\in \Diff(\C, 0)$ with $H_{k_0+1}\ne 0$ of the form \eqref{h} has exactly one non-degenerate
characteristic direction which is clearly fully attractive.
\end{remark}

\sm A {\sl parabolic manifold} $P$  of dimension $1\leq p\leq
m$ for $h\in\Diff(\C^m, 0)$ of the form \eqref{h} is the biholomorphic image of a simply connected
open set in $\C^p$ such that $0\in \partial P$, $f(P)\subset P$
and $\lim_{\ell\to\infty}h^{\circ \ell}(z)=0$ for all $z\in P$. If
$p=1$ the parabolic manifold is called a {\sl parabolic curve},
whereas, if $p=m$ the parabolic manifold is called a {\sl
parabolic domain}. Due to M. Abate \cite{Ab}, in $\C^2$ parabolic curves
always exist for germs of holomorphic diffeomorphisms tangent to the
identity and having an isolated fixed point at the origin, whereas in higher dimension they are known
to exist ``centered'' at non-degenerate characteristic
directions (J. \'Ecalle \cite{Ec}, M. Hakim \cite{Ha}, see also
\cite{ABT}).

On the other hand, Hakim~\cite{H} (based on the previous work by Fatou
\cite{Fa} and Ueda \cite{Ue1}, \cite{Ue2} in $\C^{2}$)  studied the
so-called {\em semi-attractive} case, with one eigenvalue equal to
$1$ and the rest of eigenvalues having modulus less than $1$. She proved that
either there exists a curve of fixed points or there exist
attracting open petals. Such a result has been later
generalized by Rivi \cite{Ri}.
The {\em quasi-parabolic} case of a germ in $\C^2$, i.e.\ having one
eigenvalue $1$ and the other of modulus equal to one, but not a
root of unity has been
studied in \cite{B-M} and it has been proved that, under a certain
generic hypothesis called ``dynamical separation'', there
exist petals tangent to the eigenspace of $1$. Such a result
has been generalized to higher dimension by Rong \cite{R1},
\cite{R2}. We refer the reader to the survey papers \cite{Ab2}
and \cite{Br} for a more accurate review of existing results.

\sm We state here the theorem of \'Ecalle and Hakim needed in
our paper.

\begin{theorem}[\'Ecalle \cite{Ec}, Hakim \cite{Ha}]\label{Ha}
Let $h\in\Diff(\C^m,0)$ be as in \eqref{h} and let $[v]$ be a fully
attractive non-degenerate characteristic direction for $h$.
Then there exist $k_0$ parabolic  domains such that
$h^{\circ j}(z)\ne0$ for all $j$ and $\lim_{j\to\infty}[h^{\circ j}(z)]=[v]$ for all fixed $z$ in one such a parabolic domain.
Moreover, if $v$ is a normalized representative of $[v]$, then
the parabolic domains can be chosen of the form
\begin{equation}\label{para-man}
M^i_{R, C} = \{(x,y)\in\C\times\C^{m-1} : x\in\Pi^{i}_R,\, \|y\|<C|x|\},
\end{equation}
where $\Pi^i_R$, $i=1,\ldots,k_{0}$, are the connected
components of the set $\Delta_R = \{x\in\C :
|x^{k_0}-\frac{1}{2R}|<\frac{1}{2R}\}$, and $R>0$ is sufficiently
large.

\end{theorem}

\section{Multi-resonant biholomorphisms}\label{3}

Given $\{\lambda_1,\ldots,
\lambda_n\}$ a set of complex numbers, recall that a {\em
resonance} is a pair $(j,L)$, where $j\in\{1,\ldots, n\}$ and
$L=(l_1,\ldots, l_n)\in\N^n$ is a multi-index with
$|L|:=\sum_{h=1}^n l_h \ge2$ such that $\lambda_j=\lambda^L$
(where $\lambda^L:=\lambda_1^{l_1}\cdots \lambda_n^{l_n}$). We
shall use the notation
$$
\res_j(\lambda) := \{Q\in\N^n : |Q|\ge 2, \lambda^Q = \lambda_j\}.
$$
With a slight abuse of notation, we denote by $e_j=(0,\ldots, 0,
1, 0,\ldots, 0)$ both the multi-index with $1$ at the $j$-th
position and $0$ elsewhere and the vector with the same entries
in $\C^n$. Note that if $Q\in\res_j(\lambda)$ then
$\lambda^{Q-e_j}=1$ and $P:=Q-e_j \in \6{N}_j$, where
$$
\6{N}_j := \6{P}_j\cup \6{M},
$$
with
$$
\6{P}_j:= \{P\in\Z^n : |P|\ge 1, p_j=-1, p_h\ge 0~\hbox{for}~h\ne j\}
$$
and
$$
\6{M} := \{P\in\N^n : |P|\ge 1\}.
$$
Moreover, if $P\in\6{N}_j$ is such that $\lambda^P =1$, then either $P\in\6{P}_j$, and so $P + e_j\in\res_j(\lambda)$, but $P+e_h\not\in\res_h(\lambda)$ for $h\ne j$ and $k P+e_j\not\in\res_j(\lambda)$ for each integer $k\ge 2$, or $P\in\6{M}$, yielding $k P + e_h\in\res_h(\lambda)$ for $h=1,\dots, n$ for any $k\in\N\setminus\{0\}$.
We shall denote
$$
\6{N} := \bigcup_{j=1}^n \6{N}_j = \6{M}\cup\bigcup_{j=1}^n \6{P}_j.
$$

In the rest of the paper, and without mentioning it
explicitly, we shall consider only germs of diffeomorphisms
whose differential is diagonal.

\begin{definition}\label{m-resonant}
Let $F$ be in $\Dif$, and let $\lambda_1,\ldots, \lambda_n$ be
the eigenvalues of the differential $dF_0$. We say that $F$ is
{\it $m$-resonant with respect to the first $r$ eigenvalues}
$\l_{1},\ldots,\l_{r}$ $(1\le r\le n)$ if there exist $m$
multi-indices $P^1,\dots,P^m\in \N^{r}\times\{0\}^{n-r}$ linearly independent
over $\Q$, so that the  resonances $(j,L)$ with $1\le j\le r$
are precisely of the form
\begin{equation}\label{L}
L = e_j + k_1 P^1 + \cdots + k_m P^m
\end{equation}
with $k_1,\dots, k_m\in\N$ and $k_{1}  + \cdots + k_m\ge 1$. The vectors $P^1,\dots, P^m$ are called {\it generators over $\N$} of the resonances of $F$ in the first $r$ coordinates.
We call $F$ {\it multi-resonant with respect to the first $r$ eigenvalues}  if it is $m$-resonant
with respect to these eigenvalues for some $1\le m\le r$.
\end{definition}

\begin{example}
Let the differential of $F\in \Diff(\C^{4},0)$ have eigenvalues $\l_{1},\ldots,\l_{4}$
such that $\l_1^3= 1$ but $\l_1\ne 1$, $\l_2 = -1$, and $\l_3^{-1}\l_4^2=1$. Then $F$ is $2$-resonant with respect to $\l_{1},\l_{2}$
with generators $P^{1}=(3,0,0,0)$, $P^{2}=(0,2,0,0)$.
On the other hand, $F$ is not multi-resonant with respect to all eigenvalues
because it has the resonance $\l_{3}=\l_{4}^{2}$ which is not of the form \eqref{L}.
% because $Q=(0,0,2,0) + e_3$ cannot be written as $k_1 P_1 + k_2 P_2 + e_3$.
\end{example}

\begin{remark}\label{P}
Note that if $P^1,\dots, P^m$ are generators over $\N$ of (all, that is with $r=n$) the resonances of $F$, each multi-index $P\in\N^{n}$ such that $\lambda^P =1$ is of the form
$$
P = k_1 P^1 + \cdots + k_m P^m
$$
with $k_1,\dots, k_m\in\N$.
\end{remark}

\bl\label{diff}
If $F$ is $m$-resonant with respect to $\l_{1},\ldots,\l_{r}$, then $\l_{j}\ne\l_{s}$ for
$1\le j\le r$ and $1\le s\le n$ with $j\ne s$.
\el

\bpf
Assume that $\l_{j}=\l_{s}$ with $j$ and $s$ as in the statement.
Denote
$$
\Sigma:=\{ k_1 P^1 + \cdots + k_m P^m : k_{j}\in\N, \; k_{1}  + \cdots + k_m\ge 1\},
$$
where $P^{1},\ldots, P^{m}$ are the generators.
Then for any $P\in\Sigma$, $(j,P+e_{j})$ is a resonance corresponding to the identity
$\l_{j}=\l^{P}\l_{j}$. Since $\l_{j}=\l_{s}$, we also have the resonance relation $\l_{j}=\l^{P}\l_{s}$,
which by Definition~\ref{m-resonant}, implies $P+e_{s}-e_{j}\in\Sigma$.
Therefore $\Sigma+(e_{s}-e_{j})\subset\Sigma$, and hence,
by induction, $\Sigma+\N(e_{s}-e_{j})\subset\Sigma$.
On the other hand, fixing any $P\in \Sigma$, we can find $k\in\N$ such that $P+k(e_{s}-e_{j})\notin\6M$, and hence $P+k(e_{s}-e_{j})\notin\Sigma$. We thus obtain a contradiction proving the lemma.
\epf

Let $P, Q\in\6{N}$. We  write $P<Q$ if either $|P|< |Q|$, or
$|P|=|Q|$ but $P$ precedes $Q$ in the lexicographic order,
{\sl i.e.}, $p_h = q_h$ for $1\le h< k\le n$ and $p_k<q_k$.

\me We have the following uniqueness property.

\begin{proposition}\label{unico}
Let $F\in\Dif$ be $m$-resonant. Then the set of generators is unique (up to reordering).
\end{proposition}

\begin{proof}
Let us now assume by contradiction that there are two sets of generators, $P^1,\dots,P^m$
and $Q^1,\dots,Q^m$, that we can assume to be ordered, {\sl i.e.},  $
P^1<\cdots<P^m$ and $Q^1<\cdots<Q^m$. For each $j\in\{1,\dots, m\}$ we then have
$$
P^j = k_{j,1}Q^1+\cdots + k_{j,m}Q^m,
$$
with $k_{j, s}\in\N$ and at least one $k_{j, s}\ne 0$, say
$k_{j,s(j)}$. Analogously, for $h=1,\dots, m$ we have
$$
Q^h = l_{h,1}P^1+\cdots + l_{h,m}P^m,
$$
with $l_{h, t}\in\N$, and with at least one of them non zero.
Therefore we have
$$
\begin{aligned}
P^j
&= k_{j,1}Q^1+\cdots + k_{j,m}Q^m\\
&= k_{j,1}\left(l_{1,1}P^1+\cdots + l_{1,m}P^m\right) + \cdots+ k_{j,m}\left(l_{m,1}P^1+\cdots + l_{m,m}P^m\right)\\
&= (k_{j,1}l_{1,1} + \cdots + k_{j,m}l_{m,1})P^1 + \cdots + (k_{j,1}l_{1,m} + \cdots + k_{j,m}l_{m,m})P^m,
\end{aligned}
$$
yielding, since $P^1,\dots, P^m$ are linearly independent over
the rationals,
\begin{equation}\label{eqone}
\begin{aligned}
&k_{j,1}l_{1,j} + \cdots + k_{j,m}l_{m,j} = 1\\
&k_{j,1}l_{1,h} + \cdots + k_{j,m}l_{m,h} = 0 \quad\hbox{for}~h\ne j.
\end{aligned}
\end{equation}
Now, since $k_{j, s},l_{h, t} \in\N$, the second equations in
\eqref{eqone} are satisfied only if, for $s=1,\dots, m$,
$k_{j,s}l_{s, h} = 0$ for each $h\ne j$; hence $l_{s(j), h} =
0$ for $h\ne j$ because $k_{j, s(j)} \ne 0$, and therefore
$Q^{s(j)} = l_{s(j),j} P^j$. This implies that in the first
equation in \eqref{eqone} we have $k_{j,s(j)}l_{s(j),j}\ne 0$,
and thus it has to be equal to $1$, yielding
$$
P^j = Q^{s(j)} \quad\hbox{for}~j=1, \dots, m.
$$
Therefore, since we are assuming $P^1<\cdots<P^m$ and
$Q^1<\cdots<Q^m$, we have $s(j)=j$ for all $j$, that is
$P^j=Q^j$, contradicting the hypothesis.
\end{proof}

\begin{definition}
The generators $P^1,\dots,P^m$ are called {\sl ordered} if $P^1<\cdots<P^m$.
\end{definition}

\sm Let $F\in\Dif$ be $m$-resonant (with respect to the first r eigenvalues), and let $P^1,\dots,P^m$ be
the ordered generators over $\N$ of its resonances. By the
Poincar\'e-Dulac theorem \cite{Ar}, we can find a tangent to
the identity (possibly) formal change of coordinates of
$(\C^n,0)$ conjugating $F$ to a germ in a (possibly formal)
{\sl Poincar\'e-Dulac normal form}, {\sl i.e.}, of the form
\begin{equation}\label{P-Dulac}
\widetilde F(z) = Dz + \sum_{s=1}^r\sum_{| K {\bf P}|\ge 2 \atop K\in\N^m} a_{K,s} z^{K {\bf P}} z_s e_s +  \sum_{s=r+1}^n R_s(z)e_s,
\end{equation}
where $D={\rm Diag}(\lambda_1,\dots, \lambda_n)$, we denote
by $K {\bf P} = \sum_{h=1}^m k_h P^h$, and $R_s(z)=O(\|z\|^2)$ for $s=r+1,\ldots,n$.

\begin{remark}\label{PD-limitato}
By the proof of the Poincar\'e-Dulac theorem it follows that,
given any $l\geq 2$, there exists a polynomial (hence
holomorphic) change of coordinates tangent to the identity in
$(\C^n,0)$ conjugating $F$ to a Poincar\'e-Dulac normal form up
to order $l$.
\end{remark}

Poincar\'e-Dulac normal forms are not unique because they
depend on the choice of the resonant part of the
normalization. However, they are all conjugate to each other and, by Lemma \ref{diff}, a conjugation $\psi$ between different Poincar\'e-Dulac normal forms of $F$, which are $m$-resonant with respect to the first $r$ eigenvalues, must have the first $r$ coordinates of the type
\begin{equation}\label{cambio-norm}
\psi_j(z) =\gamma_j z_j +O(\|z\|^2) \quad j=1,\ldots, r
\end{equation}
with $\gamma_j\neq 0$, $j=1,\ldots,r$. Hence  the following
definition is well-posed.

\begin{definition}
Let $F\in\Dif$ be $m$-resonant with respect to $\{\l_1,\ldots,\l_r\}$, and let $P^1,\dots,P^m$ be the
ordered generators over $\N$ of its resonances.  Let
$\2{F}$ be a Poincar\'e-Dulac normal form for $F$ given by
\eqref{P-Dulac}. The {\it weighted order} of $F$ is the minimal
$k_0=|K|\in\N\setminus\{0\}$ such that the coefficient $a_{K,
s}$ of $\2{F}$ is non-zero for some $1\le s\le r$.
\end{definition}

\begin{remark}
The weighted order of $F$ is $+\infty$ if and only if $F$ is formally linearizable in the first $r$ coordinates.
\end{remark}

Let $F\in\Dif$ be $m$-resonant, and let $P^1,\dots,P^m$ be the
ordered generators over $\N$ of its resonances. Write $P^j=(p_1^j,\ldots,p_r^j,0,\ldots,0)$, for $j=1,\ldots,m$. Let $k_0<\infty$ be
the  weighted order of $F$. Let $\2{F}$ be a
Poincar\'e-Dulac normal form for $F$ given by \eqref{P-Dulac}.
Then we set
$$
G(z) = Dz + \sum_{s=1}^r\sum_{|K|= k_0 \atop K\in\N^m} a_{K,s} z^{K {\bf P}} z_s e_s.
$$
Consider the map $\pi\colon (\C^n,0)\to (\C^m, 0)$ defined by
$\pi(z_1,\dots, z_n) := (z^{P^1},\dots, z^{P^m}) = (u_1,\dots,
u_m)$. Therefore we can write
$$
G(z) = Dz + \sum_{s=1}^r\sum_{|K|= k_0 \atop K\in\N^m} a_{K,s} u^K z_s e_s,
$$
and $G$ induces a unique map $\Phi\colon (\C^m, 0)\to (\C^m, 0)$ satisfying $\Phi\circ \pi = \pi\circ G$, which is
tangent to the identity of order greater than or equal to $k_0+1$, and is of the form
$$
\Phi(u) = u + H_{k_0 + 1}(u)+O(\|u\|^{k_0+2}),
$$
where
\begin{equation}\label{eqmatrixhakim}
H_{k_0 + 1}(u) =\left(\begin{array}{c}
				\displaystyle u_1\sum_{|K| = k_0} \left({p^1_1}\frac{a_{K, 1}}{\l_1} + \cdots + {p^1_r}\frac{a_{K, r}}{\l_r}\right)u^K\\
				\vdots\\
				\displaystyle u_m\sum_{|K| = k_0} \left({p^m_1}\frac{a_{K, 1}}{\l_1} + \cdots + {p^m_r}\frac{a_{K, r}}{\l_r}\right)u^K
				\end{array}\right).
\end{equation}

\begin{definition}
We call $u\mapsto u+H_{k_0+1}(u)$ a {\sl parabolic shadow} of $F$.
\end{definition}

\begin{remark}\label{well-posed}
Let $f:u\mapsto u+H_{k_0+1}(u)$ be a parabolic shadow of $F$.
Then clearly $H_{k_0+1}$  remains unchanged under (holomorphic
or formal) changes of coordinates of the type $z\mapsto
z+O(\|z\|^2)$ which preserve Poincar\'e-Dulac normal forms of
$F$. In case of a linear change of coordinates preserving
Poincar\'e-Dulac normal forms, by Lemma~\ref{diff}, it has to
be of the form $(z_{1},\ldots,z_{n})\mapsto
(\mu_{1}z_{1},\ldots,\mu_{r}z_{r}, l(z))$ with
$l(z)\in\C^{n-r}$. Then the $a_{K,j}$'s become
$a_{K,j}\mu^{K\7P}$, where
$\mu=(\mu_{1},\ldots,\mu_{r},0,\ldots, 0)\in\C^{n}$, and the
parabolic shadow becomes $\tilde{f}:u\mapsto
u+\2H_{k_{0}+1}(u)$ where
\begin{equation}\label{cambio-H}
\2H_{k_{0}+1}(u):= \sum_{j=1}^m\sum_{|K|=k_{0}} \mu^{K\7P}  ( {p^j_1}\frac{a_{K, 1}}{\l_1} + \cdots + {p^j_r}\frac{a_{K, r}}{\l_r} ) u^{K} u_{j}e_{j}.
\end{equation}
In particular, if $[v]\in \C\mathbb{P}^{m-1}$ is a non-degenerate characteristic direction of $f$
with $v\in\C^m$ a normalized representative
({\sl i.e.}, satisfying \eqref{1overk}), then
$$
\tilde v := (\mu^{-P^{1}} v_{1}, \ldots,  \mu^{-P^{m}} v_{m}),
$$
is a normalized representative of a non-degenerate characteristic direction for $\tilde{f}$. Moreover, the matrix
$A(\tilde v)$ of $\tilde{f}$ is conjugated to $A(v)$ (see \cite{Ha} and \cite{Ari})---hence, if $[v]$ is fully attractive for $f$, then so is $[\tilde v]$
for $\tilde f$.
\end{remark}

\begin{definition}\label{basic-def}
Let $F\in\Dif$ be $m$-resonant with respect to $\l_{1}, \ldots,\l_{r}$, with $P^1,\dots,P^m$ being the ordered generators over $\N$ of the resonances.
Let $k_0<+\infty$ be the weighted order of $F$. Let $f$ be a parabolic shadow of $F$.
We say that $F$ is {\sl $(f,v)$-attracting-non-degenerate} if $v\in \C^m$ is a normalized representative of a fully attractive
non-degenerate characteristic direction for $f$.

If $F$ is $(f,v)$-attracting-non-degenerate and  $f(u)= u+H_{k_0+1}(u)$ with $H_{k_0+1}$ as in \eqref{eqmatrixhakim},  we say that $F$ is
{\sl $(f,v)$-parabolically attracting} with respect to $\{\l_1,\ldots, \l_r\}$ if
\begin{equation}\label{parabolic}
\Re\Bigg(\sum_{|K|=k_0\atop K\in\N^m} {\frac{a_{K, j}}{\l_j}} v^{K}\Bigg) <0 \quad j=1,\ldots, r.
\end{equation}

We say that $F$ is {\sl attracting-non-degenerate} (resp. {\sl
parabolically attracting}) if $F$ is $(f,v)$-attracting-non-degenerate
(resp. $(f,v)$-parabolically attracting) with respect to
some parabolic shadow $f$ and some   normalized representative
$v\in \C^m$ of a fully attractive non-degenerate characteristic
direction for $f$.
\end{definition}

\begin{remark}
By Remark \ref{1-nota-fully-attractive}, if $F$ is $m$-resonant with $m=1$, then it is one-resonant as defined in \cite{BZ}; it is easy to check that $F$ is attracting-non-degenerate (resp. parabolically attracting)
according to Definition \ref{basic-def} if and only if it is so
according to the corresponding definitions in \cite{BZ}.
\end{remark}

\begin{remark}
According to Remark \ref{well-posed}, if $F$ is $(f,v)$-attracting-non-degenerate (resp. $(f,v)$-parabolically attracting) with
respect to some parabolic shadow $f$ and some normalized representative $v\in \C^m$, then it is so with respect to any parabolic shadow
(for the corresponding normalized representative  $\tilde{v}$). It might also happen, as we show in Section \ref{esempio}, that $F$ is $(f,v')$-attracting-non-degenerate (resp. $(f,v')$-parabolically attracting) with respect to the same parabolic shadow $f$ but to a different direction $v'$.
\end{remark}

\section{Dynamics of multi-resonant
maps}\label{4}

\begin{definition}
Let $F\in\Dif$. We call a {\sl (local) basin of attraction (or attracting basin) for $F$ at
$0$} a nonempty (not necessarily connected) open set
$U\subset\C^n$ with $0\in \1U$, contained in the domain of definition of $F$, which is $F$-invariant and such that $F^{\circ m}(z)\to 0$ as
$m\to \infty$  whenever $z\in U$.
\end{definition}

\begin{theorem}\label{mainthm}
Let $F\in\Dif$ be $m$-resonant with respect to the  eigenvalues
$\{\l_1,\ldots\l_r\}$ and of weighted order $k_0$. Assume that
$|\l_j|=1$ for $j=1, \dots, r$ and $|\l_j|<1$ for
$j=r+1,\ldots, n$. If $F$ is parabolically attracting, then
there exist (at least) $k_0$ disjoint basins of attraction
having $0$ at the boundary.
\end{theorem}

\begin{proof}
Let $P^1,\dots,P^m$ be the ordered generators over $\N$ of the resonances of $F$.
Up to biholomorphic conjugation, we can assume that $F(z) = (F_1(z),\dots, F_n(z))$ is of the form
$$
\begin{aligned}
&F_j(z) = \lambda_j z_j \Bigg(1 + \sum_{k_0\le | K|\le k_l \atop K\in\N^m}  \frac{a_{K,j}}{\lambda_j} z^{K {\bf P}} \Bigg) + O\left(\|z\|^{l+1}\right), \quad j=1,\ldots, r,\\
&F_{j}(z)=\l_{j} z_{j} + O(\|z\|^{2}) , \quad\quad\quad   \quad\quad\quad\quad\quad\quad \quad\quad\quad\quad j=r+1,\ldots, n,
\end{aligned}
$$
where
$$
k_l:= \max \{|K| : | {K} {\bf P}|\le l \},
$$
with $K{\bf P}:= \sum_{h=1}^m k_h P^h$, and $l> 1$ will be chosen at a later stage (see \eqref{beta}).

We consider the map $\pi\colon (\C^n,0)\to (\C^m,0)$ defined by $\pi(z) = u := (z^{P^1},\dots, z^{P^m})$. Then we can write
$$
F_j(z) =   G_{j}(u,z) + O\left(\|z\|^{l+1}\right), \quad  G_{j}(u,z):= \lambda_j z_j
\Bigg(1 + \sum_{k_0\le |K|\le k_l \atop K\in\N^m} \frac{a_{K,j}}{\l_j} u^K \Bigg),
\quad j=1,\ldots, r.
$$
The composition $\phi:= \pi\circ F$ can be written as
$$
\phi(z)=\Phi(u,z): =  \1\Phi(u) + g(z), \quad \1\Phi(u)=  u + H_{k_0 + 1}(u) + h(u),
$$
where $\Phi\colon\C^m\times\C^n\to \C^m$, $\1\Phi$ is induced by $G$ via $\pi\circ G=\1\Phi\circ \pi$, the homogeneous polynomial $H_{k_0+1}(u)$ has the form \eqref{eqmatrixhakim}, and where $h(u) = O(\|u\|^{k_0+2})$ and $g(z) = O(\|z\|^{l+1})$.

Since $F$ is attracting-non-degenerate by hypothesis, its parabolic shadow
$u\mapsto u + H_{k_0 + 1}(u)$ has a fully attractive
non-degenerate characteristic direction $[v]$, such that $v\in
\C^m$ is a normalized representative, {\sl i.e.}, $v$ satisfies
\eqref{1overk} and the real parts of the eigenvalues of
$A=A(v)$ are all positive. In particular, we can apply Theorem
\ref{Ha} to $\overline{\Phi}(u)$. Then there exist $k_0$
disjoint parabolic domains $M^i_{R,C}$, $i=1,\ldots, k_0$,
for $\1\Phi$ at $0$ in which every point is attracted to the
origin along a trajectory tangent to $[v]$. We shall use linear
coordinates $(x,y)\in \C\times \C^{m-1}$ where $v$ has the form
$v = (1,0,\dots,0)$ and where the matrix $A$ is in Jordan
normal form.

We will construct  $k_0$ basins of attraction $\2B^i_{R,
C}\subset \C^n$, $i=1,\ldots, k_0$ for $F$ in such a way that
each $\2B^i_{R, C}$ is projected into $M^i_{R,C}$ via $\pi$. The
parabolic domains $M^i_{R,C}$'s are given by
\eqref{para-man}, and we can assume that the component
$\Pi^1_R$ is chosen centered at the direction $1$. We first
construct a basin of attraction based on $M^1_{R,C}$. We
consider the sector
$$
S_R(\eps) : = \{x\in\Delta_R : |{\sf
Arg}(x)|<\eps\}\subset \Pi^1_R,
$$
for some $\eps>0$ small to be chosen later, and we let
\[
B^1_{R, C}(\eps) = \{(x,y)\in\C\times\C^{m-1} : x\in S_R(\eps),\, \|y\|<C|x|\}.
\]
Let $\beta>0$  and let
$$
\widetilde B:= \{z\in\C^n : |z_j|< |x|^\beta ~\hbox{for}~j=1,\dots, n, \; u = \pi (z)\in B^1_{R, C}(\eps) , \; u = (x, y)\in\C\times \C^{m-1}\},
$$

First of all, taking $\beta>0$ sufficiently small, since $x$ is
a linear combination of $z^{P^1},\ldots, z^{P^m}$, it is easy
to see that $\widetilde B$ is  an open non-empty set  of $\C^n$
and $0\in\partial \widetilde B$.

Next, we prove that $\widetilde B$ is $F$-invariant. Let $z\in
\widetilde B$ and let $u=\pi(z)$. We consider the blow-up
$\widetilde \C^m$ of $\C^m$ at the origin, in the local chart
centered at $v$, where the projection $\widetilde \C^m\mapsto
\C^m$ is given by $(x,\hat y)\mapsto
(x,y)=(x,x\hat y)$. Since $d\1\Phi_0=\id$, the map $\Phi$ can be
lifted to $\widetilde \C^m\times\C^n$ with values in $\widetilde \C^m$. In the coordinates $\hat u=(x,\hat
y)$, the lifting of the map $\Phi$ to $\widetilde \C^m\times \C^n$ takes then
the form \cite{Ha, Ari}:
$$
\begin{aligned}
&\3\Phi_1(x,\hat y,z) = x - \frac{1}{k_0} x^{k_0 + 1} + h_1(x,x\hat y) + g_1(z),\\
&\3\Phi'(x,\hat y,z) = (I - x^{k_{0}} A)\hat y + O(|x|^{k_0+1}) + O(|x|^{-1}\|z\|^{l+1}),
\end{aligned}
$$
where we denote by $\3\Phi'(u,z) = (\3\Phi_2(u,z),\dots, \3\Phi_m(u,z))$.

We now consider the change of coordinates $(U,\hat y)= (x^{-k_0},\hat y)$. In these new coordinates we have
$$
\begin{aligned}
&\widetilde\Phi_1(U,\hat y,z) =\Phi_1(U^{-1/k_0}, \hat y, z)^{-k_0},\\
&\widetilde\Phi'(U,\hat y,z) = \3\Phi'(U^{-1/k_0}, \hat y, z).
\end{aligned}
$$
Note that $x\in S_R(\eps)$ if and only if $U\in H_R(\eps)$, where
$$
H_R(\eps):=\{w\in \C: \Re w>R , |{\sf Arg}(w)|<k_0\eps\}.
$$

Now fix $0<\delta<1/2$ and $0<c'<c$. Since $\Re \left(\sum_{|K|
= k_0} \frac{{a_{K,j}}}{\l_j}v^K\right)<0$ by the parabolically attracting hypothesis, there exists $\eps>0$ such that
\begin{equation}\label{inc-vert}
\left|1+ \sum_{|K|= k_0 \atop K\in\N^m} \frac{a_{K,j}}{\l_j}  u^{K} \right| \le 1 - 2{c }|x|^{k_{0}},
\end{equation}
for all $u\in B^1_{R,C}(\eps)$. We choose $\beta>0$ such that
\begin{equation}\label{betabis}
\beta(\delta +1 ) -c'k_0<0
\end{equation}
and we  choose $l>1$ so that
\begin{equation}\label{beta}
\beta l > k_0 + 1.
\end{equation}

We have
$$
\begin{aligned}
\widetilde\Phi_1(U,\hat y,z)
&=U  \left(\frac{1}{1 - \frac{1}{k_0 U} + U^{1/k_0}h_1(U^{-1/k_0},y) + U^{1/k_0} g_1(z)}\right)^{k_0}.
\end{aligned}
$$
Since in $\widetilde B$ we have $\|z\|\le n|x|^\beta$, and
$\|u\|\le |x| + \|y\|\le (1+C) |x|$, recalling that $h_1(u)=
O(\|u\|^{k_0+2})$ and $g_1(z) = O(\|z\|^{l+1})$, we obtain that
there exists $K>0$ such that
$$
|U|^{1/k_0}{|h_1(U^{-1/k_0},y)|} \leq K |U|^{1/k_0} |U|^{-(k_0+2)/k_0} = K |U|^{-1-1/k_{0}},
$$
and
$$
|U|^{1/k_0}{|g_1(z)|}\leq K |U|^{1/k_0} |x|^{\beta (l+1)}=K |U|^{(1-\beta (l+1))/k_0}.
$$
Therefore, by \eqref{beta}, if  $R$ is sufficiently large, for
any $z\in \widetilde B$ (and hence for any $U\in H_R(\eps)$), we
have
\begin{equation}\label{nu}
\2\Phi_1(U,\hat y,z) = U+1+\nu(U,\hat y,z)
\end{equation}
with $|\nu(U,\hat y,z)|<\delta<1/2$. In particular,
$U_1:=\2\Phi_1(U,\hat y,z)\in H_R(\eps)$. Therefore we
proved that (under a suitable choice  of $R$)
\begin{equation}
z\in \widetilde B \Rightarrow x_1:=\Phi_1(u,z)\in S_R(\eps).
\end{equation}

Next we check that $\hat y_1:=\3\Phi'(x,\hat y,z)$ satisfies $\|\hat y_1\|< C$. With the same argument as in \cite{Ha}, we obtain for suitable $K>0$:
$$
\begin{aligned}
\|\hat y_1\|
&\le \|\hat y\|(1 -\lambda |x|^{k_{0}}) +K|x|^{k_0+1}+ K|x|^{-1}\|z\|^{l+1}\\
&\le \|\hat y\|(1 -\lambda |x|^{k_{0}}) +K|x|^{k_0+1}+ K|x|^{\beta(l+1)-1}\\
&\le C(1 -\lambda |x|^{k_{0}}) +K|x|^{k_0+1}+ K|x|^{\beta(l+1)-1}\\
&\leq C
\end{aligned}
$$
where  $\lambda>0$ is such that the real parts of the
eigenvalues of $A=A(v)$ are all greater than $\lambda$, and so
$\|\hat y_1\|\leq C$  for $R$ sufficiently large, in view of
\eqref{beta}. Hence we proved that if $z\in \widetilde B$, then
\begin{equation}\label{pi}
\pi(F(z))=\Phi(u, z)\in B^{1}_{R,C}(\eps).
\end{equation}

Now, given $z\in \widetilde B$, we have to estimate $|F_j(z)|$ for $j=1,\dots, n$.
We first examine the components  $F_j$ for $j=r+1,\ldots, n$.
Set $z':=(z_1,\ldots, z_r)$ and $z'':=(z_{r+1},\ldots, z_n)$.
Then
\[
F(z)''=M z''+h(z) z,
\]
where $M$ is the $(n-r)\times (n-r)$ diagonal matrix with
entries $\lambda_j$ ($j=r+1,\ldots, n$) and $h$ is a
holomorphic $(n-r)\times n$ matrix valued function in a
neighborhood of $0$ such that $h(0)=0$. If $z\in \2B$, then
$|z_{j}|<|x|^\beta$ for all $j$. Moreover, since $|\lambda_j|<1$ for
$j=r+1,\ldots, n$, it follows that there exists $a<1$ such that
$|\l_{j}z_{j}|<a|z_{j}|<a|x|^\beta$ for $j=r+1,\ldots, n$. Also, let $0<b<1-a$. Then, for $R$
sufficiently large, it follows that $\|h(z)\|\leq b/n$ if $z\in
\2B$. Hence, letting $p=a+b<1$, we obtain
\begin{equation}
\label{p1}
|F_{j}(z)|\leq |\l_{j} z_{j}|+\|h(z)\| \|z\|<a|x|^\beta+\frac{b}{n} n|x|^\beta= (a+b)|x|^\beta=p |x|^\beta.
\end{equation}
Now, we claim that for $R$ sufficiently large, it follows that
\begin{equation}\label{plimio}
|x|\leq \frac{1}{p^{1/\beta}}|x_1|,
\end{equation}
where $x_1=\Phi_{1}(u,z)$. Indeed, \eqref{plimio} is equivalent to
$|U_1|\leq p^{-k_{0}/\beta} |U|$ and hence to
\[
\frac{|U+1+\nu(U,\hat y,z)|}{|U|}\leq p^{-k_{0}/\beta}.
\]
But the limit for $|U|\to \infty$ in the left-hand side is $1$
and the right-hand side is $>1$, thus \eqref{plimio} holds
for $R$ sufficiently large.
Hence, by \eqref{p1} and \eqref{plimio} we obtain
\begin{equation}\label{p2}
|F_{j}(z)|<  |x_1|^\beta, \quad j=r+1,\ldots,n.
\end{equation}

For the other coordinates, we have
$$
F_j(z) = \lambda_j z_j \Bigg(1+ \sum_{|K|= k_0 \atop K\in\N^m} \frac{a_{K,j}}{\l_j}  u^{K} + f_j(u) \Bigg)+ g_j(z),
\quad f_{j}=O(\|u\|^{k_{0}+1}), \quad g_{j}=O(\|z\|^{l+1}),
$$
for $j=1,\ldots,r$.
Thanks to \eqref{inc-vert} we have
$$
\left| 1+ \sum_{|K|= k_0 \atop K\in\N^m} \frac{a_{K,j}}{\l_j}  u^{K} + f_j(u)\right|  \le 1 - {c }|x|^{k_{0}}.
$$
Moreover, if $z\in \widetilde B$ and $R$ is sufficiently large, we have, for a suitable $D>0$,
$$
|g_j(z)|\le D\|z\|^{l+1} < D |x|^{\b(l+1)}.
$$
Therefore for $j=1,\ldots, r$
\begin{equation}\label{Fj}
\begin{aligned}
|F_j(z)|
&\le |\lambda_j| |x|^\b \left( 1 - \frac{c }{|U|} \right) +D |x|^{\b(l+1)}, \\
&\le  |x|^\b \left( 1 - \frac{c }{|U|}\right)+ \frac{D}{|U|^{\b l/k_0}} |x|^\b\\
&\le \left( 1 - \frac{c }{|U|}+ \frac{D}{|U|^{\b l/k_0}} \right)|x|^\b.
\end{aligned}
\end{equation}
Since we have chosen $\beta l > k_0+1$ in \eqref{beta}, we get $\beta l/k_{0}>1$. Hence, if $R$ is
sufficiently large, for all $U\in H_R(\eps)$
$$
p(U) : = 1 - \frac{c }{|U|}+ \frac{D}{|U|^{\b l/k_0}} <1.
$$
Now we claim that, setting $x_1 =  \Phi_1(u,z)$, we get
$$
|x|\le p(U)^{-1/\beta}|x_1|,
$$
which is equivalent to
\begin{equation}\label{stimadue}
|U_1| \le p(U)^{-k_0/\beta}|U|.
\end{equation}
%Using Hakim's computations {\bf insert citation} (see also {\bf [Abate-Duke]}), that can be performed identically in $\C^m$, we get
Since, by \eqref{nu}, we have
$$
U_1 = U + 1 + \nu (U, \hat y, z),
$$
with $|\nu(U, \hat y, z)|\le \delta$, we obtain
$$
\frac{|U_1|}{|U|} = \frac{|U + 1 + \nu (U,\hat y, z)|}{|U|}
\le  1 + \frac{1}{|U|} + \frac{|\nu(U,\hat y,z)|}{|U|}
\le 1 + \frac{1+ \delta }{|U|}.
$$
On the other hand, by our choice of $0<c'<c$ and taking $R$
sufficiently large, we have
$$
\left(1-\frac{c'}{|U|}\right)^{-k_0/\beta} \le p(U)^{-k_0/\beta},
$$
and hence, in order to prove \eqref{stimadue}, we just need to
check that
$$
1 + \frac{1+ \delta }{|U|} \le \left(1-\frac{c'}{|U|}\right)^{-k_0/\beta}.
$$
But
$$
\left(1-\frac{c'}{|U|}\right)^{-k_0/\beta} = 1 + \frac{k_0}{\beta}\frac{c'}{|U|} + O\left(\frac{1}{|U|^{2}}\right),
$$
and since \eqref{betabis} ensures that $\delta + 1 -c'k_0/\beta<0$, if $R$ is sufficiently large, \eqref{stimadue} holds, and the claim is proved.
Therefore, in view of \eqref{Fj} we have
\begin{equation}\label{}
|F_j(z)| < |x_1|^\beta, \quad j=1,\dots, r,
\end{equation}
which together with \eqref{p2} and \eqref{pi} implies
 $F(\widetilde B)\subseteq \widetilde B$.

Then, setting inductively $u^{(l)}=(x^{(l)},y^{(l)}): = \pi(F^{\circ (l-1)}(z))$, and denoting by $\rho_j\colon\C^n\to\C$ the projection $\rho_j(z)= z_j$, we obtain
$$
\left|\rho_j\circ F^{\circ l}(z)\right|\le \left|x^{(l)}\right|^\beta
$$
 for all $z\in \widetilde B$.
Moreover, as a consequence of \eqref{nu},  we have $\lim_{l\to
+ \infty}x^{(l)} = 0$, implying that $F^{\circ l}(z)\to 0$ as
$l \to +\infty$. This proves that $\widetilde B$ is a basin of
attraction of $F$ at $0$.

Finally, since the same argument can be repeated for each of
the parabolic domains  $M^i_{R,C}$ of $\overline\Phi$, and
those are disjoint, we obtain at least $k_0$ disjoint basins of
attraction, and this concludes the proof.
\end{proof}

\subsection{Fatou coordinate for $m$-resonant parabolically attracting germs}

In this subsection we shall prove the existence of the so-called {Fatou coordinate} for $m$-resonant parabolically at\-tracting germs, ending the proof of Theorem \ref{mainintro}. Given $F\in\Diff(\C^n, 0)$ with $B$ attracting basin of parabolic type, a holomorphic function $\psi\colon B\to \C$ so that $F$ is semi-conjugated to a translation, {\sl i.e.},
\begin{equation}\label{Fatou}
\psi\circ F(z) = \psi(z) +1
\end{equation}
for all $z\in B$, is usually called the \emph{Fatou coordinate} of $F$ relative to $B$. In dimension one, given $f\in \Diff(\C,0)$ tangent to the identity, $f^q\not\equiv\id$ for any $q\ge 1$, there exists the Fatou coordinate of $f$ relative to $P$, for any attracting petal $P$ (see for example \cite[Theorem~3.2]{Ab2}), and in higher dimension Hakim in \cite{Ha1} (see also \cite{Ari}) proved the existence of the Fatou coordinate for any tangent to the identity germ with a fully attractive non-degenerate characteristic direction relative to the associated attracting basin (see \cite{V} for more recent results also in the degenerate case).

\begin{proposition}\label{fatou_m_res}
Let $F\in\Dif$ be $m$-resonant with respect to the first $r\le n$ eigenvalues. Assume that $|\l_j|=1$ for $j=1,\ldots, r$, $|\l_j|<1$ for $j=r+1,\dots, n$, and suppose that $F$ is parabolically attracting. Then for each attracting basin $B$ of~$F$ given by Theorem \ref{mainthm}, there exists a Fatou coordinate $\mu\colon B\to \C$.
\end{proposition}

\begin{proof}
Let $B$ be one of the attracting basin of $F$ constructed as in Theorem \ref{mainthm}, and use the same notation as in the proof of Theorem \ref{mainthm}. Denoting by $U_\ell$ the $U$-coordinate of $\pi(F^{\circ\ell}(z))$ for $\ell\ge 0$, it follows from the proof of Theorem \ref{mainthm} that there exists $c\in\C$ such that
\begin{equation}\label{Phi}
U_{\ell+1} = U_\ell + 1 + \frac{c}{U_\ell} + O\left(|U_\ell|^{-2}, |U_\ell|^{-(1+\beta)}\right).
\end{equation}
Consider the sequence $\{\mu_\ell\}$, with $\mu_\ell\colon B\to \C$, defined as
$$
\mu_\ell (z) = U_\ell - \ell - c \log\frac{1}{U_\ell}.
$$
We claim that $\{\mu_\ell\}$ is a Cauchy sequence with respect to the  topology of uniform convergence, and therefore it has a holomorphic limit function~$\mu$. Indeed, we have
$$
\begin{aligned}
\mu_{\ell +1} (z) - \mu_\ell(z)
&= U_{\ell+1} - \ell-1 -c \log U_{\ell +1} - U_\ell + \ell + c \log U_{\ell}\cr
&= U_\ell + 1 +  \frac{c}{U_{\ell}}  - 1- U_\ell - c \log\left(\frac{ U_\ell + 1 +  \frac{c}{U_{\ell}}}{U_\ell}\right)+O\left(|U_\ell|^{-2}, |U_\ell|^{-(1+\beta)}\right)\cr
&= \frac{c}{U_{\ell}} - c\log\left(1 + \frac{1}{U_\ell} + \frac{c}{U_{\ell}^2} \right)+ O\left(|U_\ell|^{-2}, |U_\ell|^{-(1+\beta)}\right).
\end{aligned}
$$
Hence for all $z\in B$, we have
$$
|\mu_{\ell +1} (z) - \mu_\ell(z) | = O\left(|U_\ell|^{-2}, |U_\ell|^{-(1+\beta)}\right),
$$
and therefore
$$
\mu_\ell - \mu_0 = \sum_{j=0}^\ell (\mu_{j+1}-\mu_j)
$$
converges absolutely uniformly on $B$ to a holomorphic limit $\mu-\mu_0$. The limit function $\mu$ semi-conjugates $F$ to a translation. Indeed, we have
$$
\begin{aligned}
\mu(F(z))=\lim_{\ell\to \infty}\mu_\ell (F(z))= \lim_{\ell\to \infty} \left[U_{\ell+1} -\ell - c\log\frac{1}{U_{\ell +1}}\right] = \lim_{\ell\to \infty} \mu_{\ell+1}(z) + 1 = \mu(z) + 1,
\end{aligned}
$$
and we are done.
\end{proof}

\section{Final remarks}\label{final}

\subsection{Example of germs parabolically attracting with respect to different directions}\label{esempio}

Here we construct a family of  $2$-resonant germs in $\C^3$ which are $(f,v_1)$-attracting-non-degenerate and $(f,v_2)$-attracting-non-degenerate (where $f$ is a parabolic shadow  and $v_1, v_2$ are normalized representatives  of two different fully attractive non-degenerate characteristic directions) and which are $(f,v_1)$-parabolically attractive but not  $(f,v_2)$-parabolically attractive.

Let $P^1=(2,3,0)$ and $P^2=(0,2,5)$. Let $\l_1,\l_2,\l_3\in \C^\ast$ be of modulus $1$ such that relations are generated by $\l_1^2\l_2^3=1$ and $\l_2^2\l_3^5=1$. It is easy to see that  $\l_j=\l^L$ for $L\in \N^3$, $|L|\geq 2$, $j=1,2,3$ if and only if $L=k_1P^1+k_2P^2+e_j$ for some $k_1,k_2\in \N$.

Let $F$ be of the form
\begin{equation*}
F_j(z)=\l_j z_j (1+b_{e_1,j}z^{P^1}+b_{e_2,j}z^{P^2})\quad j=1,2,3.
\end{equation*}

Then $F$ is $2$-resonant with respect to $\{\l_1,\l_2,\l_3\}$ and of weighted order $k_0=1$. A parabolic shadow of $F$ is $f(u)= u+H_2(u)$ where
\[
H_2(u)=\left(\begin{array}{c}
				\displaystyle u_1\left[ (2b_{e_1,1} + 3b_{e_1,2})u_1+(2b_{e_2,1} + 3b_{e_2,2})u_2\right]\\
				\displaystyle u_2\left[ (2b_{e_1,2} + 5b_{e_1,3})u_1+(2b_{e_2,2} + 5b_{e_2,3})u_2\right]
				\end{array}\right).
\]
The directions $[1:0]$ and $[0:1]$ are characteristic directions. Imposing
\begin{equation}\label{dima1}
2b_{e_1,1} + 3b_{e_1,2}=-1, \quad 2b_{e_2,2} + 5b_{e_2,3}=-1,
\end{equation}
it follows that the two directions are non-degenerate characteristic directions for $f$. Furthermore setting
\begin{equation}\label{dima2}
2b_{e_2,1} + 3b_{e_2,2}=-p,\quad 2b_{e_1,2} + 5b_{e_1,3}=-q
\end{equation}
with $q,p>1$ it is easy to see that $(1,0)$ and $(0,1)$ are normalized representative of fully attractive non-degenerate characteristic directions for $f$---hence $F$ is $(f,(1,0))$-attracting-non-degenerate and
$(f,(0,1))$-attracting-non-degenerate.

Finally, $F$ is $(f,(1,0))$-parabolically attractive if and only if
\begin{equation}\label{dima3}
\Re b_{e_1,j}<0 \quad j=1,2,3
\end{equation}
whereas $F$ is $(f,(0,1))$-parabolically attractive if and only if
\begin{equation}\label{dima4}
\Re b_{e_2,j}<0 \quad j=1,2,3.
\end{equation}

Given $p,q>1$, for any $b_{e_t,j}$, $t=1,2$, $j=1,2,3$ such that \eqref{dima1}, \eqref{dima2}, \eqref{dima3} are satisfied and $\Re b_{e_2,1}>0$  (such set of solutions is not empty, as it can be easily checked) the corresponding map $F$ is $(f,(1,0))$-parabolically attracting and
$(f,(0,1))$-attracting-non-degenerate but not $(f,(0,1))$-parabolically attracting.

\subsection{Basins of attraction for the inverse of an $m$-resonant germ}\label{inversem}

\begin{proposition}\label{attr-rep}
Let $F\in\Dif$ be $m$-resonant with respect to   $\{\l_1,\ldots,\l_r\}$ and of weighted order $k_0$. Then $F^{-1}$ is $m$-resonant with respect to   $\{\l_1^{-1},\ldots,\l_r^{-1}\}$ and of weighted order $k_0$. Moreover, if $F$ is attracting-non-degenerate (resp. parabolically attracting) with respect to   $\{\l_1,\ldots,\l_r\}$ then $F^{-1}$ is
attracting-non-degenerate (resp. parabolically attracting) with respect to   $\{\l_1^{-1},\ldots,\l_r^{-1}\}$.
\end{proposition}

\begin{proof}
The eigenvalues of the linear part of the germ $F^{-1}$ are  $\{\l_1^{-1},\ldots,\l_n^{-1}\}$.
By Remark \ref{PD-limitato} we can choose local holomorphic coordinates such that
$$
F_j(z) = \lambda_j z_j + \sum_{k_0\le | K|\le k_l \atop K\in\N^m}  a_{K,j} z^{K {\bf P}} z_j + O\left(\|z\|^{l+1}\right), \quad j=1,\ldots, r.
$$
Therefore
$$
F_j^{-1}(z) = \lambda_j^{-1} z_j  - \sum_{|K|=k_0} \frac{a_{K, j}}{\l_j^{2}} z^{K {\bf P}} z_j + \cdots, \quad j=1,\ldots, r.
$$
Hence $F^{-1}$   is $m$-resonant with respect to    $\{\l_1^{-1},\ldots,\l_r^{-1}\}$ and of weighted order $k_0$.

Moreover, assume that $F$ is $(f,v)$-attracting-non-degenerate with respect to some parabolic shadow $f(u)=u+H_{k_0 +1}(u)$ of $F$ and a normalized representative $v$ of a fully attractive non-degenerate characteristic direction for $f$. The corresponding parabolic shadow of $F^{-1}$ is $\2f: u\mapsto u-H_{k_0 +1}(u)$. Therefore, for any $\zeta\in\C$ so that $\zeta^{k_0} = -1$, we have
$$
-H_{k_0 +1}(\zeta v) = -\frac{1}{k_0}(\zeta v).
$$
Since the map induced by $-H_{k_0 +1}$ in $\C\P^{m-1}$ is the same as the one induced by $H_{k_0 +1}$, the matrix $A(\zeta v)$ has all eigenvalues with positive real parts also in this case. Hence $F^{-1}$ is $(\2f,\zeta v)$-attracting-non-degenerate.

Moreover, if $F$ is $(f,v)$-parabolically attracting then
$$
\Re\Bigg(\sum_{|K|=k_0\atop K\in\N^m} \frac{-a_{K, j}/\l_j^2}{\l_j^{-1}} (\zeta v)^{K} \Bigg)= \Re \Bigg(\sum_{|K|=k_0\atop K\in\N^m} \frac{-a_{K, j}}{\l_j} \zeta^{k_0} v^{K} \Bigg) = \Re \Bigg(\sum_{|K|=k_0\atop K\in\N^m} \frac{a_{K, j}}{\l_j} v^{K} \Bigg)< 0.
$$
Hence $F^{-1}$ is $(\2f,\zeta v)$-parabolically attracting.
\end{proof}

If $F\in\Dif$, as customary a basin of attraction for $F^{-1}$ is called a {\sl repelling basin for $F$}. From Proposition \ref{attr-rep} and Theorem \ref{mainthm} we thus have the following corollary:

\begin{corollary}
Let $F\in \Dif$ be  $m$-resonant with respect to all eigenvalues  $\{\l_1,\ldots,\l_n\}$ and of weighted order $k_0$. Assume that $|\l_j|=1$ for all $j=1,\ldots, n$. If $F$ is parabolically attracting  then  there exist (at least) $k_0$ repelling basins for $F$ having $0$ on the boundary.
\end{corollary}

\subsection{Leau-Fatou flower theorem for  one-resonant Poincar\'e-Dulac normal form }

Let $G\in\Dif$ be one-resonant with respect to all eigenvalues $\{\l_1,\ldots,\l_n\}$, with $|\l_j|=1$, $j=1,\ldots, n$, and assume it is in Poincar\'e-Dulac normal form \eqref{P-Dulac}. Assume $G$ is non-degenerate, parabolically attracting (cfr. Remark \ref{1-nota-fully-attractive}), with order $k_0\ge 1$ and generator $\a\in \N^n$. Let $\pi:\C^n\to \C$ be the projection given by
$(z_1,\ldots, z_n)\mapsto z^\a$. Let $u=z^\a=\pi(z)$. Set
\begin{equation}\label{phi-u}
\Phi(u):=G_1(z)^{\a_1}\cdots G_n(z)^{\a_n}=u+\Lambda(G) u^{k_0+1} +O(|u|^{k_0+2}).
\end{equation}
 Note that $\Phi:(\C,0)\to (\C,0)$ is
tangent to the identity and $\pi \circ G=\Phi \circ \pi$.  Let $v^+_1,\ldots, v^+_{k_0}$ be the attracting directions for $\Phi$,
and let $P^+_j\subset\C$, $j=1,\ldots,k_0,$ be the attracting
petal centered at $v^+_j$. Arguing as in the proof of \cite[Proposition 4.2]{BZ}, since $\Phi$ has no terms depending on $z$, it is not difficult to show that the sets
\[
U^+_j:=\pi^{-1}(P^+_j)\subset\C^{n},
\]
are basins of attraction for $G$.

On the other hand, by (the proof of) Proposition \ref{attr-rep}, if $v^-_1,\ldots, v^-_{k_0}$ are the repelling directions for $\Phi$, and $P^-_j\subset\C$, $j=1,\ldots,k_0,$ is the repelling
petal centered at $v^-_j$, the sets
\[
U^-_j:=\pi^{-1}(P^-_j)\subset\C^{n}
\]
are repelling basins of  $G$.  By the very construction, the union of the $U^+_j$'s and the $U^-_j$'s and $\displaystyle\bigcup_{j=1\atop \a_j\ne 0}^n\{z_j=0\}$ covers a full neighborhood of the origin.

Since on $\{z_j = 0\}$ for $\a_j\neq 0$ the map $G$ is linear, this provides a complete dynamical picture of $G$ in a full neighborhood of $0$. Hence we have the following generalization of the Leau-Fatou flower theorem:

\begin{theorem}\label{L-Fatou}
Let $F\in\Dif$ be one-resonant with respect to all eigenvalues $\{\l_1,\ldots,\l_n\}$ with generator $\a\in \N^n$. Assume that $F$ is holomorphically conjugated to one of its Poincar\'e-Dulac normal forms. Suppose that $|\l_j|=1$, $j=1,\ldots, n$ and  $F$ is parabolically attracting. Then for each $j\in \{1,\ldots, n\}$ such that $\a_j\neq 0$ there exists a germ $M_j$ of a complex manifold  tangent to $\{z_j=0\}$ at $0$ such that $F(M_j)\subset M_j$ and $F|_{M_j}$ is holomorphically linearizable. Moreover, there exists  an open neighborhood $W$ of $0$ such that $W\setminus \bigcup_j M_j$ is the union of  attracting and repelling basins of~$F$.
\end{theorem}

\bibliographystyle{alpha}

\end{document}